\documentclass[reqno, 12pt]{amsart}
\makeatletter
\let\origsection=\section \def\section{\@ifstar{\origsection*}{\mysection}} 
\def\mysection{\@startsection{section}{1}\z@{.7\linespacing\@plus\linespacing}{.5\linespacing}{\normalfont\scshape\centering\S}}
\makeatother        

\usepackage{amsmath,amssymb,amsthm}
\usepackage{mathrsfs}
\usepackage{mathabx}\changenotsign

\usepackage{xcolor}
\usepackage[backref=page]{hyperref}
\hypersetup{
    colorlinks,
    linkcolor={red!60!black},
    citecolor={green!60!black},
    urlcolor={blue!60!black}
}

\usepackage[abbrev,msc-links,backrefs]{amsrefs}
\usepackage[T1]{fontenc}
\usepackage{lmodern}

\usepackage[english]{babel}

\usepackage{geometry}
\usepackage{enumitem}

\theoremstyle{plain}
\newtheorem{thm}{Theorem}[section]

\newtheorem{prop}[thm]{Proposition}

\theoremstyle{definition}

\newtheorem{quest}{Question}

\let\eps=\varepsilon
\let\theta=\vartheta
\let\rho=\varrho
\let\phi=\varphi

\def\EE{\mathbb E}
\def\NN{\mathbb N}
 
\def\PP{\mathbb P}

\def\cF{{\mathcal F}}

\def\cC{{\mathcal C}}

\DeclareMathOperator{\forb}{Forb}

\let\polishlcross=\l
\def\l{\ifmmode\ell\else\polishlcross\fi}

\def\dcup{\dot\cup}

\def\colond{\colon\,}

\def\tand{\ \text{and}\ }
\def\qand{\quad\text{and}\quad}

\def\homa{\xrightarrow{\text{hom}\,}}
\let\phi=\varphi

\let\setminus=\smallsetminus
\let\emptyset=\varnothing

\begin{document}
\title[Structure of dense graphs with fixed clique number]{On 
the structure of dense graphs with fixed clique number}

\author[Heiner Oberkampf]{Heiner Oberkampf}
\address{Institut f\"ur Informatik, Universit\"at Augsburg,  Augsburg, Germany}
	\email{heiner.oberkampf@informatik.uni-augsburg.de}

\author[Mathias Schacht]{Mathias Schacht}
\address{Fachbereich Mathematik, Universit\"at Hamburg, Hamburg, Germany}
\email{schacht@math.uni-hamburg.de}
\thanks{The second author was supported through the Heisenberg-Programme of the DFG}


\begin{abstract}
We study structural properties of graphs with fixed clique number 
and high minimum degree. In particular, we  show that there exists a 
function $L=L(r,\eps)$, such that every $K_r$-free graph $G$ 
on~$n$ vertices with minimum degree at least $(\frac{2r-5}{2r-3}+\eps)n$ 
is homomorphic to a $K_r$-free graph on at most $L$ vertices. It is known that 
the required minimum degree condition is approximately best possible for this result.

For $r=3$ this result was obtained by \L uczak 
[\emph{On the structure of triangle-free graphs of large minimum degree}, Combinatorica~\textbf{26} (2006), no.~4, 489--493] 
and, more recently,
Goddard and Lyle [\emph{Dense graphs with small clique number}, J.~Graph Theory~\textbf{66} (2011), no.~4, 319-331] 
deduced the general case from \L uczak's result. \L uczak's proof was based on an application of 
Szemer\'edi's regularity lemma and, as a consequence, it only gave rise to a tower-type bound on $L(3,\eps)$. The proof 
presented here replaces the application of the regularity lemma by a probabilistic argument, which yields a bound for~$L(r,\eps)$
that is doubly  exponential  in $\textrm{poly}(\eps)$.
\end{abstract} 

\maketitle

  														  \section{Introduction}						  														  

\subsection{Notation}
We follow the notation from~\cite{RD} and briefly review some of it below. 
The graphs we consider here are undirected, simple, and have no loops and for 
a graph $G=(V,E)$ we denote by \mbox{$V=V(G)$} its \emph{vertex set} and by $E=E(G)\subseteq\binom{V}{2}$ 
its \emph{edge set}. The number of vertices is finite and often denoted by $n=|V|$. By $C_r$ and $K_r$
we denote the \emph{cycle} and the \emph{complete graph/clique} on $\l$ vertices. 
For two adjacent vertices $x, y$ we simply denote its edge by  $xy$.
If $x\in X\subseteq V(G)$ and $y\in Y\subseteq V(G)$ then $xy$ is an \emph{$X-Y$-edge}.
For disjoint sets $X$ and $Y$ the set of all $X-Y$-edges is a subset of $E$ and 
is denoted by $E_G(X,Y)$. Moreover, the number of $X-Y$ edges in $G$ is denoted by $e_G(X,Y)=|E_G(X,Y)|$.
and for nonempty~$X$ and $Y$ the  \emph{density} is defined by $d_G(X,Y)=\frac{e_G(X,Y)}{|X||Y|}$.
The set of \emph{neighbours} of a vertex $v$ is denoted by $N_G(v)$ and its size
$d_G(v)=|N_G(v)|$ is the \emph{degree} of a vertex $v$, where we sometimes suppress the subscript $G$ if there is no danger of confusion. We denote the  \emph{minimum degree} of $G$ by $\delta(G)$.
For a subset $U\subseteq V$ we define the \emph{common (or joint) neighbourhood} of $U$ as
\[
	N(U)= \bigcap_{u \in U} N(u)
\]
and we emphasise  that this differs from the notation in~\cite{RD}.
For later reference we note that the size of $N(U)$ can be easily bounded from below 
in terms of the minimum degree of $G=(V,E)$ by
\begin{equation}\label{eq:k-vertices}
	|N(U)|\geq |U|\cdot\delta(G)-(|U|-1)\cdot |V|\,.
\end{equation}
Moreover, for a subset $U\subseteq V$ we denote by $G[U]$ 
the \emph{induced subgraph on $U$} and we write $G - U$ for $G[V\setminus U]$.

A (vertex) \emph{colouring} of a graph $G = (V,E)$ is a map $c\colond V \rightarrow \cC$ such that $c(v) \neq c(w)$, whenever $v$ and $w$ are adjacent.
The elements of the set~$\cC$ are called the available \emph{colours}.
A graph $G=(V,E)$ is \emph{$k$-colourable} if there exists colouring $c\colond V \rightarrow [k]= \{1,\dots,k\}$.
The \emph{chromatic number} $\chi(G)$
is the smallest integer $k$ such that $G$ has is $k$-colourable.

A \emph{homomorphism} from a graph $G$ into a graph  $H$ is a mapping
$\phi\colond V(G) \rightarrow V(H)$ which preserves adjacencies, i.e.\
for all pairs $uv \in E(G)$ we have
$\phi(u)\phi(v)\in E(H)$.
We write 
\[
	G\homa H
\] 
to indicate that some homomorphism~$\phi$ exists.
We say $G$ is a \emph{blow-up} of $H$, if $G$ is obtained from $H$
by replacing every vertex $x$ of $H$ by an independent sets~$I_x$ 
and edges of $H$ correspond to complete bipartite graphs, i.e., 
$I_x$, $I_y$ spans a complete bipartite graph in $G$ if $xy\in E(H)$
and otherwise $e_G(I_x,I_y)=0$. If $G$ is a blow-up of $H$
then $G\homa H$ and $K_r\subseteq G$ if and only if $K_r\subseteq H$.

By $F\subseteq G$ we mean that $G$ contains a copy of $F$,
that is there exists an injective homomorphism from $F$ into $G$. If $G$ contains no such copy 
of $F$ (i.e., $F\nsubseteq G$), then we say $G$ is $F$-free.
For a graph~$F$ 
\[
	\forb(F)=\{G\colond F \not\subseteq G\}
\]
denotes the class 
of $F$-free graphs, i.e.\
the collection of those graphs $G$ which do not contain a copy of $F$. Moreover, we set
\[
	\forb_n(F)=\{G\colond F \not\subseteq G \tand |V(G)|=n\}\,.
\]
An $F$-free graph~$G=(V,E)$ 
is \emph{maximal $F$-free} if the addition of any edge to $G$ leads to a copy of $F$ in $G$, i.e.,
for every $xy\in \binom{V}{2}\setminus E$ we have $F\subseteq (V,E\cup\{xy\})$.
Finally, we define the join $G\vee H$ of two graphs~$G$ and $H$ as the graph with 
vertex set $V(G\vee H) = V(G)\cup V(H)$
and edge set
\[
	E(G\vee F) = E(G)\cup E(F) \cup \{vw\colond v\in V(G) \tand w\in V(F)\}\,. 
\]

\subsection{Chromatic thresholds of graphs}
We are interested in structural properties of graphs $G\in \forb(F)$, 
which are forced by an additional minimum degree assumption on~$G$.
For example, if $F$ is a clique and $\delta(G)$ is sufficiently high, 
then Tur\'an's theorem~\cite{Tu41} assert that~$G$ is $(r-1)$-partite and, in particular,
the chromatic number of $G$ is bounded by a constant independent of $|V(G)|$.
More generally, Andr\'asfai~\cite{A62}
raised the following question:
\emph{For a given graph $F$ and an integer $k$, what is smallest condition imposed on the minimum degree~$\delta(G)$ 
such any graph $G\in\forb(F)$ satisfying this minimum degree condition has chromatic number at most~$k$?}
Here we are interested in the case when the minimum degree condition yields an upper bound on $\chi(G)$
independent from the graph~$G$ itself. 
This leads to the so called \emph{chromatic threshold} for a given \mbox{graph $F$}
\[
	\delta_{\chi}(F) 
	= 
	\inf\big\{\alpha\in[0,1]\colond \exists k\in\NN \text{ s.t.\ }
		\chi(G)\leq k \ \ \forall G\in \forb(F)\text{ with }\delta(G)\geq \alpha |V(G)|\big\}\,.
\]
If $F'\subseteq F$, then $\forb(F')\subseteq\forb(F)$, so obviously $\delta_{\chi}(F') \leq \delta_{\chi}(F)$.
Moreover, it follows from the Erd\H os--Stone theorem~\cite{ErSt46} that 
$\delta_{\chi}(F)
	\leq 
	\frac{\chi(F)-2}{\chi(F)-1}$
for every graph $F$ with at least one edge.

For $F=K_3$ it was shown in~\cite{ES73} 
that $\delta_{\chi}(K_3)\geq 1/3$. In the other direction, Thomassen~\cite{T_C02} obtained a matching upper bound and,
therefore, $\delta_{\chi}(K_3)=1/3$.
In fact, Erd\H{o}s and Simonovits~\cite{ES73} asked whether all triangle-free graphs $G$
with $\delta(G)\geq(1/3+o(1))|V(G)|$ are 3-colorable. 
This was answered negatively by H\"aggkvist~\cite{H82}, but recently Brandt and Thomass\'e~\cite{BT}  showed that the 
chromatic number of such graphs is bounded by~$4$. 

Nikiforov~\cite{N10} and Goddard and Lyle~\cite{GL10} extended those results from triangles to $r$-cliques
and showed for every $r\geq3$ that
\begin{equation}\label{eq:dchir}
	\delta_{\chi}(K_r) = \frac{2r-5}{2r-3}
\end{equation}
and, in fact,  $\chi(G)\leq r+1$ for every 
$K_r$-free graph $G$ with $\delta(G)>\frac{2r-5}{2r-3}|V(G)|$.

In the case when $F$ is an odd cycle of length at least five it was shown by Thomassen~\cite{T_C07} that 
the chromatic threshold is zero and \L uczak and Thomass\'e~\cite{LT10} proved that 
$\delta_{\chi}(F)\not\in(0,1/3)$ for all graphs~$F$ and that $\delta_{\chi}(F)=0$ if $F$ is nearly bipartite 
(a graph is nearly bipartite if it is triangle-free and it admits a vertex partition into two parts
such that one part is independent and the other part induces a graph with maximum degree one).
Recently, Allen et al.~\cite{ABK} extended of the work of \L uczak and Thomass\'e and determined 
the chromatic threshold for every graph~$F$.

\subsection{Homomorphism thresholds of graphs}
Viewing $\chi(G)\leq k$ as the property that $G\homa K_k$, one may ask
for a graph $G\in\forb(F)$, whether $K_k$ can be replaced by a graph $H$ of bounded size (independent of $G$), that is 
$F$-free itself.
More precisely, in \cite{T_C02} Thomassen posed the following question: \emph{Given a fixed constant $c$, does there exist a finite family of triangle-free graphs such that every triangle-free graph on $n$ vertices and minimum degree greater than $cn$ is homomorphic to some graph of this family?}
To formalise this question we define the \emph{homomorphism threshold} of a graph $F$
\begin{multline*}									
\delta_{\hom}(F) = 
	\inf\big\{\alpha\in[0,1]\colond\exists\, k\in\NN \text{ s.t.\ } 
		\forall\, G\in \forb(F)\text{ with } \delta(G)\geq\alpha|V(G)|\\
		\exists\, H\in \forb_k(F)\text{ with } G\homa H\big\}\,.
\end{multline*}				
Thomassen then asked to determine $\delta_{\hom}(K_3)$.
Since $G\homa H$, implies  $\chi(G)\leq |V(H)|$,  we clearly have
									\[\delta_{\hom}(F)\geq\delta_{\chi}(F).\]
									In \cite{L06} \L uczak proved $\delta_{\hom}(K_3)=1/3$ and, hence, 
for the  triangle $K_3$ the homomorphic and the chromatic threshold are equal.
Recently, Goddard and Lyle \cite{GL10} extended \L uczak's result
showing, that $K_r$-free graphs with minimum degree bigger than $\frac{2r-5}{2r-3}$ are homomorphic to the join
$K_{r-3} \vee H$, where $H$ is a triangle-free graph with $\delta(H)>|V(H)|/3$.
Consequently, we have for every $r\geq 3$
									\begin{equation}\label{eq:homchi}
	\delta_{\hom}(K_r)=\delta_{\chi}(K_r)=\frac{2r-5}{2r-3}\,.
\end{equation}
									\L uczak's proof in~\cite{L06} was based on Szemer\'edi's regularity lemma~\cite{Sz}.
We give a different proof of~\eqref{eq:homchi}, which avoids the regularity lemma 
and uses a simple probabilistic argument.

\begin{thm}\label{theorem:main}
	For every integer $r\geq 3$ we have
	\[\delta_{\hom}(K_r)=\frac{2r-5}{2r-3}.\]
\end{thm}

It seems an interesting open question to determine the homomorphism threshold
for other graphs than cliques. In particular, the case of 
odd cycles of length at least five  seems to be a first interesting open case and 
we put forward the following question.

\begin{quest}\label{q:1}
	What is $\delta_{\hom}(C_{2\l+1})$ for $\l\geq 2$?
\end{quest}

A somewhat related question concerns the homomorphism threshold 
for forbidden families of graphs. Note that the definitions of 
$\forb(F)$ and $\delta_{\hom}(F)$ straight forwardly extend from
one forbidden graph $F$ to forbidden families $\cF$ of graphs.
In view of Question~\ref{q:1} it is natural to consider the family 
$\cC_{2\l+1}=\{C_3,\dots,C_{2\l+1}\}$ of odd cycles 
of length at most $2\l+1$ and we close this introduction with the following open
question.

\begin{quest}\label{q:2}
	What is $\delta_{\hom}(\cC_{2\l+1})$ for $\l\geq 2$?
\end{quest}

																															\section{Simple observations}																																										
For an integer $r\geq 3$ and $\eps>0$ 
the following subclass of $\forb(K_r)$ will play a prominent r\^ole
\[
	\cF(r,\eps)
	=
	\big\{G=(V,E)\in\forb(K_r) \colond \delta(G)\geq \big(\tfrac{2r-5}{2r-3}+\eps \big)|V| \big\}\,,
\]
since Theorem~\ref{theorem:main} asserts that there exists some function $L=L(r,\eps)$ and 
$H\in\forb(L,K_r)$ such that for every $G\in \cF(r,\eps)$ we have $G\homa H$. Note that 
$\cF(r,\eps)$ contains only graphs on at least $r-2$ vertices.
We begin with a few observations concerning common neighbourhoods 
in maximal $K_r$-free graphs of given minimum degree $\delta(G)$.

\begin{prop}\label{prop:2nonadj}
	For $r\geq3$ let $G=(V,E)$ be a maximal $K_r$-free graph.
	If two distinct vertices $u$, $v\in V$ are non-adjacent, then $|N(u)\cap N(v)|\geq r\delta(G)-(r-2)|V|$.
\end{prop}
\begin{proof}
	Since $G=(V,E)$ is maximal $K_r$-free and $uv\not\in E$, the joint neighbourhood $N(u)\cap N(v)$ induces a
	$K_{r-2}$. Applying~\eqref{eq:k-vertices} to the $r-2$  vertices $w_1,\dots,w_{r-2}$ that span $K_{r-2}$ 
	in the joint neighbourhood of $u$ and $v$ yields $N(\{w_1,\dots,w_{r-2}\})\geq (r-2)\delta(G)-(r-3)|V|$.
	Moreover, since  $N(\{w_1,\dots,w_{r-2}\})$ must be disjoint from $N(u)\cup N(v)$, we obtain
	\begin{align*}
	|V| &\geq(r-2)\delta(G)-(r-3)|V| + |N(u)\cup N(v)| \\
		&= (r-2)\delta(G)-(r-3)|V| + |N(u)| + |N(v)|- |N(u)\cap N(v)| 
	\end{align*}
	and the proposition follows.
\end{proof}

In the proof of the last proposition we used the observation, that the neighbourhood 
of any  two non-adjacent vertices in a maximal $K_r$-free graph induces a $K_{r-2}$.
Next we note that for maximal $K_r$-free graphs in $\cF(r,\eps)$, we can strengthen 
this observation and ensure that the clique $K_{r-2}$ is disjoint from an arbitrary given small 
set of vertices.

\begin{prop}\label{prop:K_r-2}
	For $r\geq 3$ and $\eps>0$,
	let $G=(V,E)$ be a maximal $K_r$-free graph from~$\cF(r,\eps)$.
	If two distinct vertices $u$, $v\in V$ are non-adjacent in $G$ and $U\subseteq V$ satisfies $|U|<\eps |V|$, 
	then $K_{r-2}\subseteq G[(N(u)\cap N(v))\setminus U]$.
\end{prop}

\begin{proof}
	Given $u$, $v$ and $U$ as stated, we first consider any set of $r-3$ vertices $w_1,\dots,w_{r-3}\in V$ and 
	owing to~\eqref{eq:k-vertices} we have 
	\[
		N(\{w_1,\dots,w_{r-3}\})\geq (r-3)\delta(G)-(r-4)|V|\,.
	\]
	Moreover, since $u$ and $v$ are non-adjacent Proposition~\ref{prop:2nonadj} tells us that 
	\[
		|N(u)\cap N(v)|\geq r\delta(G)-(r-2)|V|\,.
	\]
	Consequently, the joint neighbourhood of $u$, $v$ and $w_1,\dots,w_{r-3}$ satisfies 
	\begin{align*}
		|N(\{u,v,w_1,\dots,w_{r-3})| & \geq N(\{w_1,\dots,w_{r-3}\}) - (|V|-|N(u)\cap N(v)|)\\
			& \geq (2r-3)\delta(G) -(2r-5)|V|
	\end{align*}
	and the minimum degree condition from $G\in \cF(r,\eps)$ implies that 
	\[
		|N(\{u,v,w_1,\dots,w_{r-3})|\geq (2r-3)\eps |V|\geq 3\eps|V| >|U|\,.
	\]
	Summarising, we have shown that any collection of $r-3$ vertices together with $u$ and $v$
	have a joint neighbour outside of $U$. Selecting  
	$w_1$ from $(N(u)\cap N(v))\setminus U$ and inductively $w_{i+1}$
	from $N(\{u,v,w_1,\dots,w_i)\setminus U$
	for $i=1,\dots,r-3$ yields the desired clique on $w_1,\dots,w_{r-2}$.
\end{proof}

The last observation asserts that any sufficiently large subset of vertices 
induces a $K_{r-2}$ in a graph $G$ from  $\cF(r,\eps)$.

\begin{prop}\label{prop:K_{r-2}inY}
	For $r\geq 3$ and $\eps>0$ let $G=(V,E$) be a graph from $\cF(r,\eps)$. If $Z\subseteq V$ satisfies
	$|Z|\geq(\frac{2r-6}{2r-3}+\eps)|V|$,
	then  $K_{r-2}\subseteq G[Z]$.
\end{prop}
\begin{proof}
	Similarly as in the proof of Proposition~\ref{prop:K_r-2} we consider an arbitrary set of 
	$(r-3)$ vertices $w_1,\dots,w_{r-3}\in V$ and 
	from~\eqref{eq:k-vertices} we infer 
	\begin{align*}
		|N(\{w_1,\dots,w_{r-3}\})\cap Z|
		&\geq 
		(r-3)\delta(G)-(r-4)|V|-(|V|-|Z|)\\
		&\geq
		\Big((r-3)\frac{2r-5}{2r-3}+\frac{2r-6}{2r-3}-(r-3)\Big)|V|+(r-2)\eps |V|\\
		&=(r-2)\eps |V|>0\,.
	\end{align*}
	Consequently, any set of $k-3$ vertices has a joint neighbour in~$Z$. Hence, selecting $w_1$ in $Z$
	and inductively $w_{i+1}$ from $N(\{w_1,\dots,w_i\})\cap Z$ for $i=1,\dots,r-3$ yields 
	the desired clique on $w_1,\dots,w_{r-2}$.
\end{proof}

																															\section{Proof of the main result}																																										
	In the proof of  Theorem~\ref{theorem:main} we partition the vertex set of a maximal $K_r$-free graph 
	$G\in\cF(r,\eps)$ into a bounded number of 
	stable sets and show that any two such independent sets are spanning complete or empty
	bipartite graphs. Consequently, $G$ is a blow-up of a $K_r$-free graph of bounded size, which is equivalent 
	to the property that $G$ has a $K_r$-free homomorphic image of bounded size.
	
	We obtain the independent sets in two steps:
	Roughly speaking, in the first step we consider a random subset $X \subset V(G)$ of bounded size 
	and partition the vertices of $V(G)$ according to their neighbourhood in $X$.
	However, since $X$ has only bounded size, a small (but linear sized) set of vertices may 
	have no or only a few neighbours in $X$ and we deal with those vertices in the second step, by 
	considering the neighbourhood into the independent sets from the first step. 

\begin{proof}[Proof of Theorem~\ref{theorem:main}]
	Let $r\geq 3$. Owing to~\eqref{eq:dchir} we have $\delta_{\chi}(K_r)=\frac{2r-5}{2r-3}$
	and since  by definition $\delta_{\chi}(K_r)\leq \delta_{\hom}(K_r)$, we have to prove the matching
	upper bound on $\delta_{\hom}(K_r)$. Let $\eps>0$ and set 
	\begin{equation}\label{eq:const}
		m=\lceil 4\ln(8/\eps)/\eps^2\rceil+1\,,\quad
		T=2^m\,,
		\qand 
		L=2^T+T\,.
	\end{equation}
	We will show that for any $n>L$ and for every maximal $K_r$-free graph $G=(V,E)$ from~$\cF(r,\eps)$
	there exists some $H\in\forb_L(K_r)$ such that $G\homa H$, which clearly suffices to prove the theorem.
	
		In the first part we consider a random subset $X\subseteq V$ of size $m$ chosen uniformly at random
	from all $m$-element subsets of~$V$ and we consider the random set 
	\[
		U_X
		=
		\Big\{v\in V\colond |N(v)\cap X|<\big(\tfrac{2r-5}{2r-3}+\tfrac{\eps}{2}\big)m\Big\}
	\]
	of vertices with ``small'' degree in~$X$.
	We show that with positive probability $|U_X|\leq \eps n/4$ and $|X\cap U_X|\leq \eps m/4$.
	
	It follows from Chernoff's inequality for the hypergeometric 
	distribution (see, e.g.,~\cite{random-graphs}*{Theorem~2.10, eq.~(2.6)} that for a given vertex $v\in V$ we 
	have 
	\begin{equation}\label{eq:Xv}
		\PP(v\in U_X)
		\leq
		\exp(-\eps^2m/4)\,.
	\end{equation}
	Consequently, 
	\[
		\EE[|U_X|]
		\leq 
		\exp(-\eps^2m/4)\cdot n
		\overset{\eqref{eq:const}}{<} \eps n/8
	\]
	and by Markov's inequality we have 
	\begin{equation}\label{eq:X1}
		\PP(|U_X|\leq \eps n/4)> 1/2\,.
	\end{equation}
	In other words, with probability more than $1/2$ all but at most $\eps n/4$
	vertices inherit approximately the minimum degree condition on the randomly chosen 
	set $X$. 
	
	Next we show that with probability at least $1/2$ the intersection 
	of $X$ with $U_X$ is small. This follows from a standard double counting argument.
	In fact, the same argument giving~\eqref{eq:Xv} shows that for every 
	$v\in V$ there are at most $\exp(-\eps^2(m-1)/4)\binom{n-1}{m-1}$
	different $(m-1)$-element subsets~$Y$ of~$V$ for which  
	\begin{equation}\label{eq:badY}
		|N(v)\cap Y|\leq\big(\tfrac{2r-5}{2r-3}+\tfrac{\eps}{2}\big)\cdot (m-1)\,.
	\end{equation}
	Hence, there are at most $n\exp(-\eps^2(m-1)/4)\binom{n-1}{m-1}$ pairs $(v,Y)$ 
	such that~\eqref{eq:badY}
	holds.
	Therefore, there are at most 
	\[
		\frac{n\cdot \exp(-\eps^2(m-1)/4)\binom{n-1}{m-1}}{\eps m/4}\overset{\eqref{eq:const}}{\leq}\frac{1}{2}\binom{n}{m}
	\] 
	$m$-element subsets $X\subseteq V$ that contain at least 
	$\eps m/4$ vertices $v$ such that $v$ and 
	\[
		Y=X\setminus\{v\}
	\]
	satisfy~\eqref{eq:badY}. Combining this with~\eqref{eq:X1} shows that there 
	exists an $m$-element set $X\subseteq V$ with the promised properties
	\[
		|U_X|\leq\frac{\eps}{4}n
		\qand 
		|X\cap U_X|< \frac{\eps}{4}m\,.
	\]
	
	Finally, we set 
	\[
		Y=X\setminus U_X \qand U_Y=\Big\{v\in V\colond |N(v)\cap Y|<\big(\tfrac{2r-5}{2r-3}+\tfrac{\eps}{4}\big)|Y|\Big\}
	\]
	and we note that the induced subgraph on $Y$  satisfies
	\[
		G[Y]\in\cF(r,\eps/4)
	\]
	and since $U_Y\subseteq U_X$ we also have 
	\[
		|U_Y|\leq |U_X|\leq \eps n/4\,.
	\]

	Next we define a vertex partition of $V\setminus U_Y$ given by the neighbourhoods in $Y$.
	For that we say two vertices $v$, $w\in V\setminus U_Y$ are equivalent w.r.t.~$Y$,
	if they have the same neighbours in $Y$, i.e., $N(v)\cap Y=N(w)\cap Y$. Let 
	$V_1\dcup\dots\dcup V_t=V\setminus U_Y$ be the corresponding partition 
	given by the equivalence classes and let $Y_i$ be the neighbourhood 
	of the vertices from $V_i$ in $Y$, i.e., for any $v_i\in V_i$ we have 
	\[
		N(v_i)\cap Y=Y_i\,.
	\]
	Clearly, $t\leq 2^{|Y|}\leq 2^{|X|}=2^m= T$.

	We observe that the vertex classes $V_1,\dots,V_t$ are independent sets 
	in $G$, i.e., for every $i=1,\dots,t$ we have
	\begin{equation}\label{eq:eVi} 
		E_G(V_i)=\emptyset\,.
	\end{equation}
	In fact, since every vertex $v\in V\setminus U_Y$ has at least 
	$(\frac{2r-3}{2r-5}+\eps/4)|Y|$  neighbours in $Y$ and since $G[Y]\in\cF(r,\eps/4)$ it follows from
	Proposition~\ref{prop:K_{r-2}inY} applied to $G[Y]$ and $Z=Y_i$ 
	that~$Y_i$ induces a $K_{r-2}$.
	Consequently, the $K_r$-freeness of $G$ 
	implies that no two vertices  $v_i$, $w_i\in V_i$
	can be adjacent in $G$ and~\eqref{eq:eVi} follows.
	
	Next we observe that the induced bipartite graphs given by the partition of equivalence classes 
	contain no or all edges, i.e.,
	for every $1\leq i<j\leq t$ we have 
	\begin{equation}\label{eq:eViVj} 
		e_G(V_i,V_j)=0\quad \text{or}\quad e_G(V_i,V_j)=|V_i||V_j|\,.
	\end{equation}
	Suppose for a contradiction that there are (not necessarily distinct) vertices $v_i$, $w_i\in V_i$ and 
	$v_j$, $w_j\in V_j$ such that $v_iv_j\in E(V_i,V_j)$ and $w_iw_j\not\in	E(V_i,V_j)$.
	Due to the edge $v_iv_j$ the intersection $Y_i\cap Y_j$ must be $K_{r-2}$-free and, hence, 
	in view of Proposition~\ref{prop:K_{r-2}inY} applied to~$G[Y]$ and $Z=Y_i\cap Y_j$
	we have 
	\[
		|Y_i\cap Y_j|<\Big(\frac{2r-6}{2r-3}+\frac{\eps}{4}\Big)|Y|
	\]
	and, therefore,
	\begin{equation}\label{eq:YiYj}
		|Y_i\cup Y_j|
		= 
		|Y_i|+|Y_j|-|Y_i\cap Y_j| 
		>
		\Big(2\frac{2r-5}{2r-3}-\frac{2r-6}{2r-3}+\frac{\eps}{4}\Big)|Y|
		=
		\Big(\frac{2r-4}{2r-3}+\frac{\eps}{4}\Big)|Y|\,.
	\end{equation}
	
	Next we use that $w_i\in V_i$ and $w_j\in V_j$ are non-adjacent.
	Owing to the maximality of $G$ we can apply Proposition~\ref{prop:K_r-2}
	to $G$ and $U_Y$ and obtain a clique $K_{r-2}$ outside $U_Y$ in the joint 
	neighbourhood of $w_i$ and $w_j$. Let $R$ be the vertex set of this $K_{r-2}$.
	Since $R\subseteq V\setminus U_Y$ and since the sets $V_k$ are independent for every 
	$k=1,\dots,t$ the set $R$ intersects $r-2$ classes $V_{k_1},\dots, V_{k_{r-2}}$
	different from $V_i$ and $V_j$. We consider the joint neighbourhood of $R$ in $Y$ 
	\[
		N(R)\cap Y = Y_{k_1}\cap\dots\cap Y_{k_{r-2}}
	\]
	and note that 
	\[
		|N(R)\cap Y|\geq (r-2)\Big(\frac{2r-5}{2r-3}+\frac{\eps}{4}\Big)|Y|-(r-3)|Y|=\Big(\frac{1}{2r-3}+\frac{\eps}{4}\Big)|Y|\,.
	\]
	However, combined with~\eqref{eq:YiYj} this implies that either $Y_i\cap N(R)\neq \emptyset$ or $Y_j\cap N(R)\neq \emptyset$.
	In either case this gives rise to a $K_r$ in $G$, which yields the desired contradiction and~\eqref{eq:eViVj} follows.
	
	Note that~\eqref{eq:eVi} shows that $G[V\setminus U_Y]$ is 
	homomorphic to a graph $H'$ on $t\leq T$ and it follows from~\eqref{eq:eViVj} 
	that $G[V\setminus U_Y]$ is a blow-up of $H'$. So in particular $H'$ is $K_r$-free.
	
	It remains to deal with the vertices in $U_Y$. For that we first observe that for every vertex 
	$u\in U_Y$	and $i=1,\dots,t$ we have 
	\begin{equation}\label{eq:uVi} 
		N(u)\cap V_i=\emptyset\quad \text{or}\quad N(u)\cap V_i=V_i\,.
	\end{equation}
	In fact, suppose for a contradiction, that for some $v_i$, $w_i\in V_i$ we have 
	$uv_i\in E$ while~$u$ and $w_i$ are not adjacent. Again the maximality of $G$
	and Proposition~\ref{prop:K_r-2} shows that $N(u)\cap N(w_i)$ contains a $K_{r-2}$
	avoiding $U_Y$. However, since by~\eqref{eq:eVi} and~\eqref{eq:eViVj} the vertices $v_i$ and $w_i$ 
	have the same neighbourhood in $V\setminus U_Y$ the same $K_{r-2}$ is also in the 
	neighbourhood of $v_i$, which together with $v_i$ and $u$ yields a $K_r$ in $G$.
	This contradicts $K_r\not\subseteq G$ and~\eqref{eq:uVi} follows.
	
	Next we partition $U_Y$ according to the neighbourhoods of its vertices in $V\setminus U_Y$.
	For every $S\subseteq [t]=\{1,\dots,t\}$ we set
	\[
		V_S=\bigg\{u\in U_Y\colond N(u)\setminus U_Y=\bigcup_{s\in S}V_s\bigg\}\,,
	\]
	which yields a partition of $U_Y$ into at most $2^t\leq 2^T$
	classes. Similar as in~\eqref{eq:eViVj} and~\eqref{eq:uVi} we next observe that for any $S$, 
	$S'\subseteq [t]$  with $S\neq S'$ we have 
	\begin{equation}\label{eq:eVSVS'}
		e_G(V_S,V_{S'})=0\quad \text{or}\quad e_G(V_S,V_{S'})=|V_S||V_{S'}|\,.
	\end{equation}
	The proof is very similar to the proof of~\eqref{eq:uVi}.
	Suppose for a contradiction without loss of generalisation there exist 
	vertices $v_S$, $w_S\in V_S$ and $u\in V_{S'}$ such that
	$uv_S\in E$ while $u$ and $w_S$ are not adjacent. Then by the maximality of $G$
	Proposition~\ref{prop:K_r-2} yields a  $K_{r-2}$ in $N(u)\cap N(w_S)$
	avoiding $U_Y$. Owing to~\eqref{eq:uVi} the vertices $v_S$ and $w_S$ 
	have the same neighbourhood in $V\setminus U_Y$ and, hence, the same $K_{r-2}$ 
	is also in the neighbourhood of $v_S$, which together with $v_S$ and $u$ yields a $K_r$ in $G$.
	This contradicts $K_r\not\subseteq G$ and~\eqref{eq:eVSVS'} follows.
	
	The last thing we have to show is that $V_S$ is independent in $G$, i.e., for every $S\subseteq [t]$ we
	have 
	\begin{equation}\label{eq:eVS}
		E_G(V_S)=\emptyset\,.
	\end{equation}
	This is a direct consequence of~\eqref{eq:uVi} and Proposition~\ref{prop:K_{r-2}inY}.
	In fact, it follows from~\eqref{eq:uVi} that any two vertices $u$, $v\in V_S$ 
	have the same neighbourhood in $V\setminus U_Y$. Hence, their joint neighbourhood 
	has size at least $(\frac{2r-5}{2r-3}+\frac{3\eps}{4})n$ and Proposition~\ref{prop:K_{r-2}inY}
	yields a $K_{r-2}$ in the joint neighbourhood of $u$ and $v$. Therefore, $u$ and $v$ cannot be
	adjacent in~$G$ and~\eqref{eq:eVS} follows.
	
	Summarising, we have shown that there exists a vertex partition 
	\[
		\bigcup_{i=1}^tV_i\cup\bigcup_{S\subseteq [t]} V_S=V
	\]
	of $V$ into independent sets (see~\eqref{eq:eVi} and~\eqref{eq:eVS})
	such that all naturally induced bipartite graphs are either complete 
	or empty (see~\eqref{eq:eViVj},~\eqref{eq:uVi}, and~\eqref{eq:eVSVS'}). Hence, there exists 
	a graph $H$ on $2^T+T\leq L$ vertices such that $G$ is a blow-up of $H$ and, therefore,
	$G\homa H$ and $H$ itself must be $K_r$-free.
	This concludes the proof of Theorem~\ref{theorem:main}.
\end{proof}

We remark that the size $H$ is doubly exponential in $\textrm{poly}(1/\eps)$, i.e., there exists some 
universal constant $c$ such that
\[
	|V(H)|\leq L=2^{2^{c\ln(\eps)/\eps^2}}
\]
holds.

We close by noting that the same approach used in the proof of 
Theorem~\ref{theorem:main} can be used to show Thomassen's 
result from~\cite{T_C07} that the chromatic threshold of odd cycles 
of at least five is $0$. However, it remains open, if this approach 
can also be used to address Questions~\ref{q:1} and~\ref{q:2}.

\end{document}